\documentclass{article}
\usepackage{latexsym,amssymb,amsmath,amsfonts,pictex,enumerate,amsthm}
\usepackage{ mathrsfs }
\usepackage{tikz}
\usetikzlibrary{decorations.markings}
\usepackage{scalerel,stackengine}
\usepackage[active]{srcltx}

\theoremstyle{definition}
\newtheorem{theorem}{Theorem}

\newtheorem{lemma}[theorem]{Lemma}

\stackMath
\newcommand\reallywidehat[1]{
\savestack{\tmpbox}{\stretchto{
  \scaleto{
    \scalerel*[\widthof{\ensuremath{#1}}]{\kern-.6pt\bigwedge\kern-.6pt}
    {\rule[-\textheight/2]{1ex}{\textheight}}
  }{\textheight}
}{0.5ex}}
\stackon[1pt]{#1}{\tmpbox}
}
\title{An Approximate Taylor Theorem for Analytic Lipschitz Functions}
\author{Stephen Deterding}
\date{Department of Mathematics and Physics\\ Marshall University, Huntington, WV, USA}

\begin{document}

\maketitle

\begin{abstract}
    Let $U$ be a bounded open subset of the complex plane and let $A_{\alpha}(U)$ denote the set of functions analytic on $U$ that also belong to the little Lipschitz class with Lipschitz exponent $\alpha$. It is shown that if $A_{\alpha}(U)$ admits a bounded point derivation at $x \in \partial U$, then there is an approximate Taylor Theorem for $A_{\alpha}(U)$ at $x$. This extends and generalizes known results concerning bounded point derivations.
\end{abstract}

\section{Introduction}

The behavior of a function at a boundary point is often very different from its behavior inside the set. For instance an analytic function on a bounded set is infinitely differentiable on the interior but may fail to have even a single derivative at a boundary point. However, if the structure of the set is nice enough then the functions may have a greater degree of smoothness at the boundary than would otherwise be expected. One such example of this greater degree of smoothness is the existence of a bounded point derivation at a boundary point. A bounded point derivation is a type of bounded linear functional that generalizes the concept of the derivative. Bounded point derivations have been studied in a wide variety of contexts and for a large number of function spaces.  \cite{Browder1967, Deterding2021,  Hallstrom1969, Hedberg1972, Lord1994, Wermer1967}

\bigskip

To illustrate how a bounded point derivation generalizes the derivative let $X$ be a compact subset of $\mathbb C$ and let $R_0(X)$ denote the set of rational functions with poles off $X$. Then $R(X)$ denotes the closure of $R_0(X)$ in the uniform norm and for a positive integer $t$, $R(X)$ is said to admit a $t$-th order bounded point derivation at $x \in X$ if there exists a constant $C > 0$ such that

\begin{align*}
    |f^{(t)}(x)| \leq C \Vert f\Vert _{\infty}
\end{align*}

\noindent for all $f \in R_0(X)$. Suppose that $R(X)$ admits a first order bounded point derivation at $x$ and that $f \in R(X)$ is not differentiable at $x$. Then there exists a sequence $\{f_j\} \in R_0(X)$ that converges uniformly to $f$ and since there is a bounded point derivation at $x$

\begin{align*}
    |f_j'(x) -f_k'(x)| \leq C \Vert f_j-f_k\Vert _{\infty}.
\end{align*}

\noindent Hence $\{f_j'(x)\}$ is a Cauchy sequence and thus converges and its limit may be said to be the derivative of $f$ at $x$. In this way a bounded point derivation generalizes a derivative by allowing derivatives to be defined for non-differentiable functions. It is worth mentioning that a result of Dolzhenko is that there is always a nowhere differentiable function in $R(X)$ whenever $X$ is a nowhere dense set \cite{Dolzhenko} so this construction has a practical use.

\bigskip

The primary question that this paper will focus on is the following: If $A$ is a family of functions and there is a bounded point derivation on $A$ at $x$, how close can the functions in $A$ come to being differentiable at $x$? One possibility is that all the functions in $A$ might be differentiable at $x$, in which case $x$ is called a removable singularity for $A$. More likely; however, is that the functions will have a more limited degree of smoothness at the boundary; for example, the functions may only be approximately differentiable at $x$. 

\bigskip

A set $E$ is said to have full area density at $x$ if

\begin{align*}
    \lim_{r \to 0} \frac{m(B_r(x) \cap E)}{m(B_r(x))} = 1
\end{align*}

\noindent where $B_r(x)$ is the ball centered at $x$ with radius $r$ and $m$ denotes $2$-dimensional Lebesgue measure. A function $f$ is said to be approximately differentiable at $x$ if there exists a set $E$ with full area density at $x$ and $L \in \mathbb C$ such that

\begin{align*}
    \lim_{y\to x, y \in E} \frac{f(y)-f(x)}{y-x} = L.
\end{align*}

\noindent $L$ is called the approximate derivative of $f$ at $x$.

\bigskip

Some work on these types of questions has been done previously. In \cite{O'Farrell2014} O'Farrell showed that a bounded point derivation on 
 the space of analytic Lipschitz functions $A_{\alpha}(U)$ at $x$ could be evaluated by a difference quotient formula 

\begin{align*}
    Df = \lim_{n \to \infty} \frac{f(z_n) - f(x)}{z_n -x}
\end{align*}

\noindent where $\{z_n\}$ is a sequence of points that converges non-tangentially to $x$ and $Df$ denotes the bounded point derivation at $f$. Subsequently this result was improved to 

\begin{align*}
    Df = \lim_{z \to x, z \in E} \frac{f(z) - f(x)}{z -x}
\end{align*}

\noindent where $E$ is a set with full area density at $x$ \cite{O'Farrell2016}. Similar results are known for $R(X)$ as well \cite[Corollary 3.6]{Wang1973}.

\bigskip

In this manuscript we will consider a different type of boundary smoothness property, that of an approximate Taylor theorem for analytic Lipschitz functions. The connection between bounded point derivations and approximate Taylor Theorems was first noted by Wang \cite{Wang1973} for the case of $R(X)$. Given a positive integer $t$ and a function $f \in R(X)$ we define the error at $z$ of the $t$-th degree Taylor polynomial of $f$ about $x$ by

\begin{align*}
     R_x^t f(z) = f(z) - \sum_{j=0}^t \dfrac{f^{(j)}(x)}{j!} (z-x)^j.
\end{align*}

\noindent It was shown in \cite[Theorem 3.4]{Wang1973} that if $R(X)$ admits a $t$-th order bounded point derivation at $x$ then for each $\epsilon >0 $ there exists a set $E$ with full area density at $x$ such that if $y \in E$ then for every $f \in R_0(X)$  
    
    \begin{align*}
   &  \vert R_x^t f(y) \vert \leq \epsilon \vert y-x \vert^t \Vert f \Vert_{\infty}.
    \end{align*}

\bigskip

Similar results have also been shown for $R^p(X)$, the closure of $R_0(X)$ in the $L^p$ norm, for $p \geq 2$ (When $p<2$, $R^p(X)$ coincides with $L^p(X)$ and thus does not admit any bounded point derivations or even bounded point evaluations \cite[Lemma 3.5]{Brennan1971}). For a given positive integer $t$, $R^p(X)$ is said to admit a $t$-th order bounded point derivation at $x$ if there exists a constant $C>0$ such that 

\begin{align*}
    |f^{(t)}(x)| \leq C \Vert f\Vert _{L^p(X)}
\end{align*}

\noindent for all $f \in R_0(X)$. It was shown in \cite[Theorem 4.1]{Wolf1978} that if $R^p(X)$ admits a $t$-th order bounded point derivation at $x$ then for each $\epsilon >0 $ there exists a set $E$ with full area density at $x$ such that if $y \in E$ then for every $f \in R_0(X)$  
    
    \begin{align*}
   &  \vert R_x^t f(y) \vert \leq \epsilon \vert y-x \vert^t \Vert f \Vert_{L^p(X)}.
    \end{align*}

    \bigskip

    The goal of this manuscript is to derive an approximate Taylor theorem for $A_{\alpha}(U)$ similar to the ones for $R(X)$ and $R^p(X)$. The remainder of the paper is outlined as follows. Section 2 provides background information on spaces of Lipschitz functions and bounded point derivations defined for those functions, while the set of full area density for the approximate Taylor theorem is constructed in Section 3. The final section contains the statement and proof of the approximate Taylor theorem for $A_{\alpha}(U)$.

\section{Lipschitz functions and bounded point derivatives}

Let $U$ be a non-empty bounded open subset of $\mathbb C$. For $0<\alpha<1$, a function $f$ is said to be a Lipschitz function with Lipschitz exponent $\alpha$ on $U$ if there exists a constant $k>0$ such that for all $z,w \in U$

\begin{align}
\label{lipschitz}
    |f(z)-f(w)| \leq k |z-w|^{\alpha}. 
\end{align}

\noindent For such a function the Lipschitz seminorm, denoted as $\Vert \cdot\Vert _{\text{Lip}_{\alpha}(U)}'$ is defined to be the smallest value of $k$ that satisfies \eqref{lipschitz}. It is only a seminorm because $k=0$ whenever $f$ is a constant function; however, it can be made into a norm by adding the uniform norm to the Lipschitz seminorm. Thus the Lipschitz norm with exponent $\alpha$, denoted as $\Vert \cdot\Vert _{\text{Lip}_{\alpha}(U)}$ is defined by

\begin{align*}
    \Vert f\Vert _{\text{Lip}_{\alpha}(U)} = \Vert f\Vert _{\text{Lip}_{\alpha}(U)}' + \sup_{U} |f(z)|.
\end{align*}

\noindent The space Lip$_\alpha(U)$ consists of all functions with finite Lipschitz norm on $U$. An important subspace of Lip$_\alpha(U)$ is the little Lipschitz class lip$_{\alpha}(U)$, which is the set of Lipschitz functions on $U$ such that 

\begin{equation*}
    \lim_{\delta \to 0^+} \sup_{0<|z-w|<\delta} \dfrac{|f(z)-f(w)|}{|z-w|^{\alpha}} = 0.
\end{equation*} 

\noindent $A_{\alpha}(U)$ denotes the space of lip$_\alpha(U)$ functions that are analytic on $U$. $A_{\alpha}(U)$ is said to admit a bounded point derivation of order $t$ at $x \in \partial U$ if there exists a constant $C$ such that

\begin{equation*}
    |f^{(t)}(x)| \leq C \Vert f\Vert _{\text{Lip}_{\alpha}(U)}
\end{equation*}

\bigskip

\noindent whenever $f \in$ lip$_\alpha(U)$ is analytic in a neighborhood of $U \cup \{x\}$.

Let $X = \partial U$. If $L \in $ lip$_\alpha(U)^*$ then it follows from de Leeuw's Theorem \cite[Theorem 2.1]{deLeeuw1961} that there exists a Borel-regular measure $\mu$ on $X \times X$ with no mass on the diagonal such that

\begin{align*}
    L(f) = \int_{X \times X} \frac{f(z)-f(w)}{|z-w|^{\alpha}} d\mu(z,w)
\end{align*}

\noindent whenever $f \in$ lip$_\alpha(U)$. Although $L$ is not represented directly by integration against a measure as in the cases of $R(X)$ and $R^p(X)$, the representation is close enough to a measure that many of the same techniques for representing measures can still be applied. In particular whenever $L$ is a bounded point derivation, this kind of integral representation can always be found.

\bigskip

The existence of a $t$-th order bounded point derivation on $A_{\alpha}(U)$ at $x$ implies the existence of all bounded point derivations of lower orders at $x$. In addition it provides an integral formula that represents evaluation at $y \in \mathbb C$ as long as a certain integral exists and is non-zero.

\begin{theorem}
\label{yrep}
    Suppose that $\mu$ represents a $t$-th order bounded point derivation on $A_{\alpha}(U)$ at $x$ and $y \in \mathbb C$ and define

    \begin{align*}
        c(y) = \int_{X\times X} \frac{(z-x)^{t+1}(z-y)^{-1} - (w-x)^{t+1}(w-y)^{-1}}{|z-w|^{\alpha}} d\mu(z,w).
    \end{align*}
    
   \noindent If $c(y)$ exists and is nonzero, then for all $f \in A_{\alpha}(U)$

    \begin{align*}
        f(y) = c(y)^{-1} \int_{X\times X} \frac{(z-x)^{t+1}(z-y)^{-1}f(z) - (w-x)^{t+1}(w-y)^{-1}f(w)}{|z-w|^{\alpha}} d\mu(z,w)
    \end{align*}
    
\end{theorem}

\begin{proof}
    Since $f \in A_{\alpha}(U)$ so is $\frac{f(z)-f(y)}{z-y}$ and hence

\begin{align*}
    \int_{X\times X} \frac{(z-x)^{t+1}(z-y)^{-1}(f(z)-f(y))- (w-x)^{t+1}(w-y)^{-1}(f(w)-f(y))}{|z-w|^{\alpha}} d\mu(z,w) = 0.
\end{align*}

\noindent Therefore,

\begin{align*}
    \int_{X\times X} \frac{(z-x)^{t+1}(z-y)^{-1}f(z)- (w-x)^{t+1}(w-y)^{-1}f(w)}{|z-w|^{\alpha}} d\mu(z,w) = f(y) c(y)
\end{align*}

\noindent and hence

\begin{align*}
        f(y) = c(y)^{-1} \int_{X\times X} \frac{(z-x)^{t+1}(z-y)^{-1}f(z)- (w-x)^{t+1}(w-y)^{-1}f(w)}{|z-w|^{\alpha}} d\mu(z,w).
    \end{align*}
    
\end{proof}

\section{A set with full area density at $x$}

In this section we construct the set with full area density at $x$ that will be used in the approximate Taylor Theorem for $A_{\alpha}(U)$. To prove that the set has full area density at $x$ we need the following lemma.

\begin{lemma}
\label{fullareadensity}
    Let $\mu$ be a measure on $\mathbb C \times \mathbb C$. Let $a$, $b$, and $c$ be positive real numbers with $b+c \leq a$ and $b+c < 2$, $x \in \mathbb C$ and let $\delta>0$. Then the set 

\begin{align*}
  \left\{ y \in \mathbb C : \int \frac{|y-x|^{a}}{|w-y|^b |z-y|^c} d \mu(z,w) < \delta \right\}
\end{align*}

\noindent has full area density at $x$.
\end{lemma}

 Similar results to Lemma \ref{fullareadensity} can be found in \cite{O'Farrell2016}. Specifically Lemma 3.8 with $u = 0$ is the same as Lemma \ref{fullareadensity} with $b+c = a$ and Lemma 3.3 implies the same result; however, we will provide a more direct proof than the ones found in \cite{O'Farrell2016}. We start with the following lemma.

\begin{lemma}
\label{fn}
    Let $a$, $b$, and $c$ be positive real numbers with $b+c \leq a$ and $b+c < 2$. Let $x\in \mathbb C$, let $\Delta_n = \{y\in \mathbb C: |y-x| < \frac{1}{n}\}$ and let $m$ denote 2 dimensional Lebesgue measure. Define a function $f_n(z,w)$ by  

    \begin{align*}
        f_n(z,w) = \frac{1}{m(\Delta_n)}\int_{\Delta_n} \frac{|y-x|^{a}}{|w-y|^b |z-y|^c} dm(y).
    \end{align*}

    \noindent Then $f_n(z,w)$ is uniformly bounded for all $z$, $w$ and $n$ and converges to $0$ almost everywhere ($dm$) as $n \to \infty$.  
\end{lemma}

\begin{proof}
    There are 4 possibilities: 1) $w=z=x$; 2) $w=x$, $z\neq x$; 3) $z=x$, $w\neq x$; and 4) $z\neq x$, $w \neq x$. Of these only the last one occurs on a set with positive 2 dimensional Lebesgue measure and thus it is enough to show that $f_n(z,w)$ converges to $0$ as $n \to \infty$ in this case. Since $z \neq x$ and $w \neq x$ there exists $N>0$ so that $|z-x| > n^{-1}$ and $|w-x|> n^{-1}$ for $n > N$. It thus follows from the reverse triangle inequality that for $n$ sufficiently large 

\begin{align*}
    |z-y| > | |z-x| - |y-x\Vert  > |z-x|-n^{-1}
\end{align*}

\noindent and similarly $|w-y| > |w-x| - n^{-1}$. Finally if $y \in \Delta_n$ then $|y-x| < n^{-1}$. Thus for $n$ sufficiently large

\begin{align*}
    f_n(z,w) \leq n^{-a} (|w-x|- n^{-1})^{-b} (|z-x|- n^{-1})^{-c}
\end{align*}

\noindent which tends to $0$ as $n \to \infty$.

\bigskip

To prove that $f_n(z,w)$ is uniformly bounded for all $z$, $w$ and $n$ we first rewrite the function by observing that $m(\Delta_n) = \pi n^{-2}$ and $|y-x|^a < n^{-a}$ for $y \in \Delta_n$. Thus

\begin{align*}
    f_n(z,w) = \frac{n^{2-a}}{\pi} \int_{\Delta_n} \frac{1}{|w-y|^b |z-y|^c} dm(y).
\end{align*}

\noindent Now this integral would be larger if $w=z$ and the integral was taken over $D(z,n^{-1})$, the disk centered at $z$ with radius $n^{-1}$. Thus it suffices to bound

\begin{align*}
    \frac{n^{2-a}}{\pi} \int_{D(z,n^{-1})} \frac{1}{|z-y|^{b+c}} dm(y).
\end{align*}

\noindent A computation shows that the above expression evaluates to $\frac{2n^{b+c-a}}{2-b-c}$ and thus $f_n(z,w)$ is uniformly bounded for all $z$, $w$ and $n$.
    
\end{proof}

\begin{lemma}
\label{fn2}
Let $\Delta_n = \{y\in \mathbb C: |y-x|< n^{-1}\}$ and let $\mu$ be a measure on $\mathbb C \times \mathbb C$. Also let $a$, $b$, and $c$ be positive real numbers with $b+c \leq a$ and $b+c < 2$. If $m$ denotes $2$-dimensional Lebesgue measure then 

\begin{align*}
    \frac{1}{m(\Delta_n)} \int_{\Delta_n} \int \frac{|y-x|^a}{|w-y|^b |z-y|^c} d\mu(z,w) dm(y) \to 0
\end{align*}

\noindent as $n \to \infty$.
\end{lemma}

\begin{proof}
    Let $f_n(z,w)$ be the function from Lemma \ref{fn}. Since it is bounded and tends to $0$ as $n\to \infty$ it follows from the Dominated Convergence Theorem that

\begin{align*}
    \lim_{n\to \infty} \int f_n(z,w) d\mu(z,w) = 0.
\end{align*}

\noindent Using Fubini's Theorem to switch the order of integration yields

\begin{align*}
    \frac{1}{m(\Delta_n)} \int_{\Delta_n} \int \frac{|y-x|^a}{|w-y|^b |z-y|^c} d\mu(z,w) dm(y) \to 0
\end{align*}

\noindent as $n \to \infty$, as desired.
\end{proof}

\bigskip

We can now prove Lemma \ref{fullareadensity}.

\begin{proof}
    Let $\Delta_n = \{y \in \mathbb C: |y-x| < n^{-1}\}$ and let $E_{\delta}$ denote the set

\begin{align*}
    \left\{ y \in \mathbb C:  \int \frac{|y-x|^a}{|w-y|^b |z-y|^c} d\mu(z,w) < \delta \right\}. 
\end{align*}

Then 

\begin{align*}
    \frac{1}{m(\Delta_n)} \int_{\Delta_n \setminus E_\delta} \int \frac{|y-x|^a}{|w-y|^b |z-y|^c} d\mu(z,w) dm(y) \geq \frac{\delta m(\Delta_n \setminus E_{\delta})}{m(\Delta_n)}.
\end{align*}

\noindent By Lemma \ref{fn2} the left hand side tends to $0$ as $n \to \infty$ and hence $E_\delta$ has full area density at $x$. Thus Lemma \ref{fullareadensity} is proved.   
\end{proof}

\bigskip

The set in the following theorem will be the set of full area density used in the proof of the approximate Taylor theorem for $A_{\alpha}(U)$. We first verify that it has full area density at $x$.

\begin{theorem}
\label{exceptional set}
 Suppose $\mu$ is a measure on $\mathbb C \times \mathbb C$, $t$ is a positive integer and $0 < \alpha < 1$. Let $\delta >0$ and let $E$ denote the union of the following $6$ sets:

    \begin{enumerate}
        \item $\{y \in \mathbb C: \int \frac{|y-x|}{|z-y|} d\mu(z,w) < \delta\}$
        \bigskip
        \item $\{y \in \mathbb C: \int \frac{|y-x|^{1+\alpha}}{|w-y|^{\alpha} |z-y| } d\mu(z,w) < \delta\}$
        \bigskip
        \item $\{y \in \mathbb C: \int \frac{|y-x|}{|w-y|^{1-\alpha}|z-y|^{\alpha}} d\mu(z,w) < \delta\}$
        \bigskip
        \item $\{y \in \mathbb C: \int \frac{|y-x|^{1+\alpha}}{|w-y| |z-y|^{\alpha} } d\mu(z,w) < \delta\}$
          \bigskip
        \item $\{y \in \mathbb C: \int \frac{|y-x|^{t+1}}{|w-y|^{\alpha} |z-y| } d\mu(z,w) < \delta\}$
        \bigskip
        \item $\{y \in \mathbb C: \int \frac{|y-x|^{t+1}}{|w-y| |z-y|^{\alpha} } d\mu(z,w) < \delta\}$
    \end{enumerate}

    \noindent Then $E$ has full area density at $x$.
\end{theorem}

\begin{proof}
    By Lemma \ref{fullareadensity} each of the sets above have full area density at $x$. Thus their union is also a set with full area density at $x$.
\end{proof}

\section{An approximate Taylor Theorem for $A_{\alpha}(X)$}

Throughout this section we will refer to the following lemma several times.

\begin{lemma}
\label{fracineq}
    For complex numbers $z$, $y$ and $x$

\begin{align*}
    \frac{1}{z-y} = \sum_{j=1}^t \frac{(y-x)^{j-1}}{(z-x)^j}+ \frac{(y-x)^t}{(z-x)^t(z-y)}.
\end{align*}
\end{lemma}

\begin{proof}
    The proof is by induction. The base case follows since 

\begin{align}
\label{base case}
    \frac{1}{z-y} = \frac{1}{z-x} + \frac{y-x}{(z-x)(z-y)}.
\end{align}
    
     Now suppose the hypothesis holds for $t$ and apply \eqref{base case} to it to obtain

     \begin{align*}
         \frac{1}{z-y} = \sum_{j=1}^{t+1} \frac{(y-x)^{j-1}}{(z-x)^j}+ \frac{(y-x)^{t+1}}{(z-x)^{t+1}(z-y)}
     \end{align*}

     \noindent which completes the induction.
\end{proof}

\bigskip

We will also need the following bound on the $c(y)$ term from the integral representation for evaluation at $y$ (Theorem \ref{yrep}) in order to prove the approximate Taylor theorem for $A_{\alpha}(U)$.

\begin{theorem}
\label{c(y)}
Let $U$ be an open subset of $\mathbb C$ and let $X = \partial U$. Let $\mu$ be a measure on $X \times X$ that represents a $t$-th order bounded point derivation at $x$ and let

    \begin{align*}
        c(y) = \int \frac{(z-x)^{t+1}(z-y)^{-1}-(w-x)^{t+1}(w-y)^{-1}}{|z-w|^{\alpha}} d\mu(z,w).
    \end{align*}

\noindent Then for each $\delta>0$ there exists a set $E$ with full area density at $x$ such that 

\begin{align*}
    t! -\delta < |c(y)| < t! + \delta.
\end{align*}

\end{theorem}

\begin{proof}
    It follows from Lemma \ref{fracineq} that

\begin{align*}
    c(y) = \int |z-w|^{-\alpha} \left( \sum_{j=1}^t (y-x)^{j-1} \left[ (z-x)^{t+1-j} -(w-x)^{t+1-j} \right] + \frac{(y-x)^{t+1}}{z-y} - \frac{(y-x)^{t+1}}{w-y}  \right) d\mu 
\end{align*}

\noindent Since $\mu$ represents derivation at $x$, inside the sum every term evaluates to $0$ except when $j=1$ in which case it evaluates to $t!$. Hence

\begin{align*}
    c(y) =  t! + \int \frac{(y-x)^{t+1}(w-z)}{(z-y)(w-y)|z-w|^{\alpha}} d\mu.
\end{align*}

\noindent Thus to prove the theorem it suffices to show that there exists a set $E$ with full area density at $x$ such that

\begin{align*}
    \int \frac{|y-x|^{t+1} |w-z|^{1-\alpha}}{|z-y\Vert w-y|} d\mu < \delta.
\end{align*}

\noindent It follows from the triangle inequality that

\begin{align*}
    \int \frac{|y-x|^{t+1} |w-z|^{1-\alpha}}{|z-y\Vert w-y|} d\mu \leq \int \frac{|y-x|^{t+1} }{|z-y|^{\alpha}|w-y|} d\mu + \int \frac{|y-x|^{t+1} }{|z-y\Vert w-y|^{\alpha}} d\mu
\end{align*}

\noindent and hence according to Theorem \ref{exceptional set} there exists a set $E$ with full area density at $x$ such that each integral on the right side of the inequality is less than $\frac{\delta}{2}$ for $y \in E$, which completes the proof.
    
\end{proof}

\bigskip

We can now prove the following approximate Taylor theorem for $A_{\alpha}(U)$.

\begin{theorem}
\label{Taylor}
    Let $U \subseteq \mathbb C$ be a bounded open set and suppose that $x \in \partial U$ admits a bounded point derivation on $A_{\alpha}(U)$ of order $t$. Then for every $\epsilon>0$ there exists a set $E$ with full area density at $x$ such that if $y \in E$ then for every $f \in A_{\alpha}(U)$ 

\begin{align*}
   |R^t_x f(y)| \leq \epsilon |y-x|^t \Vert f\Vert _{\text{Lip}_{\alpha}(U)}.  
\end{align*}
    
\end{theorem}

\begin{proof}
Suppose that $f \in A_{\alpha}(U)$ and let $g(z) = f(z)-f(x)-f'(x)(z-x)- \ldots - \frac{f^{(t)}(x)}{t!}(z-x)^t$. Then $g(x) = g'(x) = \ldots = g^{(t)}(x) = 0$, $g(y) = R^t_x f(y)$, and because there is a bounded point derivation at $x$, $\Vert g\Vert _{\text{Lip}_{\alpha}(U)} = C \Vert f\Vert _{\text{Lip}_{\alpha}(U)}$. Hence it suffices to show that for every $\epsilon>0$ there exists a set $E$ with full area density at $x$ such that $|g(y)| \leq \epsilon |y-x|^t \Vert g\Vert _{\text{Lip}_{\alpha}(U)}$.Let $X = \partial U$  and let $\mu$ be a measure on $X \times X$ that represents the $t$-th order bounded point derivation at $x$. Let $E$ denote the union of the following sets.

\begin{enumerate}
        \item $\{y \in \mathbb C: \int \frac{|y-x|}{|z-y|} d\mu(z,w) < \frac{\epsilon}{10}\}$
        \bigskip
        \item $\{y \in \mathbb C: \int \frac{|y-x|^{1+\alpha}}{|w-y|^{\alpha} |z-y| } d\mu(z,w) < \frac{\epsilon}{10}\}$
        \bigskip
        \item $\{y \in \mathbb C: \int \frac{|y-x|}{|w-y|^{1-\alpha}|z-y|^{\alpha}} d\mu(z,w) < \frac{\epsilon}{10}\}$
        \bigskip
        \item $\{y \in \mathbb C: \int \frac{|y-x|^{1+\alpha}}{|w-y| |z-y|^{\alpha} } d\mu(z,w) < \frac{\epsilon}{10}\}$
          \bigskip
        \item $\{y \in \mathbb C: \int \frac{|y-x|^{t+1}}{|w-y|^{\alpha} |z-y| } d\mu(z,w) < \frac{1}{4}\}$
        \bigskip
        \item $\{y \in \mathbb C: \int \frac{|y-x|^{t+1}}{|w-y| |z-y|^{\alpha} } d\mu(z,w) < \frac{1}{4}\}$
    \end{enumerate}

\noindent By Theorem \ref{exceptional set} $E$ has full area density at $x$. It follows from Theorem \ref{c(y)} that $c(y)$ exists and is non-zero and hence it follows from Theorem \ref{yrep} that 

\begin{align*}
    g(y) = c(y)^{-1} \int \frac{(z-x)^{t+1}(z-y)^{-1}g(z) - (w-x)^{t+1}(w-y)^{-1}g(w)}{|z-w|^{\alpha}} d\mu 
\end{align*}

\noindent and by Lemma \ref{fracineq} this simplifies to

\begin{align*}
    g(y) = c(y)^{-1} &\left( \sum_{j=1}^{t+1} (y-x)^{j-1} \int \frac{(z-x)^{t+1-j}g(z)  - (w-x)^{t+1-j}g(w)}{|z-w|^{\alpha}} d\mu \right.\\
    &\left.+ (y-x)^{t+1} \int \frac{(z-y)^{-1}g(z) - (w-y)^{-1}g(w)}{|z-w|^{\alpha}} d\mu   \right).
\end{align*}

\noindent Because $\mu$ represents a bounded point derivation of order $t$ at $x$ and $g(x) = g'(x) = \ldots = g^{(t)}(x) = 0$ it follows that all the terms in the sum vanish and hence

\begin{align*}
    g(y) = c(y)^{-1}(y-x)^{t+1}  \int \frac{(z-y)^{-1}g(z) - (w-y)^{-1}g(w)}{|z-w|^{\alpha}} d\mu.
\end{align*}

\noindent After adding and subtracting $(z-y)^{-1}g(w)$ to both sides of the numerator we obtain

\begin{align}
\label{twointegrals}
    g(y) = c(y)^{-1} \left( \int \frac{g(z)-g(w)}{|z-w|^{\alpha}} \frac{(y-x)^{t+1}}{z-y} d\mu + \int \frac{g(w)(y-x)^{t+1}(w-z)}{(z-y)(w-y)|z-w|^{\alpha}} d\mu\right).
\end{align}

\noindent Since $y \in E$ the first integral in \eqref{twointegrals} is bounded by $\frac{\epsilon}{10} |y-x|^t \Vert g\Vert _{\text{Lip}_{\alpha}(U)}'$.

\bigskip

Now we bound the second integral. Because $g(x) = 0$ it follows that $|g(w)(w-x)^{-\alpha}| \leq \Vert g\Vert _{\text{Lip}_{\alpha}(U)}'$ and thus the second integral is bounded by

\begin{align*}
    |y-x|^t \Vert g\Vert _{\text{Lip}_{\alpha}(U)}' \int \frac{|y-x|\cdot|w-x|^{\alpha}|w-z|^{1-\alpha}}{|w-y\Vert z-y|}d\mu.
\end{align*}

\noindent By the triangle inequality applied to both $|w-x|^{\alpha}$ and $|w-z|^{1-\alpha}$ this is less than 

\begin{align*}
    |y-x|^t \Vert g\Vert _{\text{Lip}_{\alpha}(U)}' \left( \int \frac{|y-x|}{|z-y|} d\mu + \int \frac{|y-x|^{1+\alpha}}{|w-y|^{\alpha} |z-y| } d\mu + \int \frac{|y-x|}{|w-y|^{1-\alpha}|z-y|^{\alpha}} d\mu + \int \frac{|y-x|^{1+\alpha}}{|w-y| |z-y|^{\alpha} } d\mu \right)
\end{align*}

and since $y \in E$ the second integral in \eqref{twointegrals} is bounded by $\frac{2\epsilon}{5} |y-x|^t \Vert g\Vert _{\text{Lip}_{\alpha}(U)}'$. Also by Theorem \ref{c(y)}

\begin{align*}
    t!-\frac{1}{2} < |c(y)| < t! +\frac{1}{2}
\end{align*}

\noindent and hence

\begin{align*}
    |g(y)| \leq  \frac{\epsilon/2}{t!-1/2} |y-x|^t \Vert g\Vert _{\text{Lip}_{\alpha}(U)}' \leq \epsilon |y-x|^t \Vert g\Vert _{\text{Lip}_{\alpha}(U)}
\end{align*}

\noindent as desired. 
\end{proof}

\bigskip

\bibliographystyle{abbrv}
\bibliography{bib}

\end{document}